\documentclass{amsart}

\usepackage[latin1]{inputenc}
\usepackage{amsfonts,amssymb,amsmath}
\usepackage{verbatim}
\input{xypic}

\newtheorem{theorem}{Theorem}

\newtheorem{lemma}[theorem]{Lemma}
\newtheorem{corollary}[theorem]{Corollary}
\theoremstyle{remark}
\newtheorem{remark}[theorem]{Remark}

\newcommand\Z{\ensuremath{\mathbf Z}}
\newcommand\Q{\ensuremath{\mathbf Q}}\newcommand\R{\ensuremath{\mathbb R}}
\newcommand\C{\ensuremath{\mathbf C}}
\renewcommand\R{\ensuremath{\mathbf R}}
\renewcommand\P{\ensuremath{\mathbf P}}
\newcommand\Qb{{\overline\Q}}
\newcommand\Kb{{\overline K}}

\newcommand{\ra}{{\rightarrow}}

\newcommand\End{\operatorname{End}}

\newcommand\Gal{\operatorname{Gal}}
\newcommand\GL{\operatorname{GL}}

\newcommand\M{\operatorname{M}}

\newcommand\PGL{\operatorname{PGL}}

\newcommand{\comp}{\begin{picture}(6,5)(-3,-2)\put(0,1){\circle{2}} \end{picture}}\def\circ{\comp}

\def\cO{{\mathcal O}}
\def\cBl{{\cO^\times_\ell\backslash B_\ell^\times/\Q_\ell^\times}}
\newcommand\acc[2]{\ensuremath{{}^{#1}\hskip-0.1ex{#2}}}

\title[Parametrization of abelian $K$-surfaces with quaternionic multiplication]{Parametrization of abelian $K$-surfaces with quaternionic multiplication}

\author{Xavier Guitart}
\address{Departament de Matem\`atica aplicada II, Universitat Polit\`ecnica de Catalunya, Jordi Girona 1-3 (Edifici Omega) 08034, Barcelona}
\email{xevi.guitart@gmail.com}

\author{Santiago Molina}
\address{Departament de Matem\`atica aplicada IV, Universitat Polit\`ecnica de Catalunya, Av. V\'ictor Balaguer, s/n. 08800 Vilanova i la Geltr\'u}
\email{santimolin@gmail.com}

\date{\today}
\begin{document}

\begin{abstract}
We prove that the  abelian $K$-surfaces whose endomorphism algebra
is an indefinite rational quaternion algebra are parametrized, up
to isogeny, by the $K$-rational points of the quotient of certain
Shimura curves by the group of their Atkin-Lehner involutions.
\vspace{0.2cm}

{\noindent \it To cite this article: X. Guitart, S. Molina, C. R. Acad. Sci.
Paris, Ser. I 347 (2009).}
\end{abstract}

\maketitle

\section{Introduction}
 \label{sec:Introduction} Let $K$ be a number field. An Abelian variety
$A/\Kb$ is called an \emph{Abelian $K$-variety} if for each
$\sigma\in\Gal(\Kb/K)$ there exists an isogeny
 $\mu_\sigma:\acc\sigma A\ra A$ such
that $\psi\circ\mu_\sigma=\mu_\sigma \circ \acc\sigma\psi $ for
all $\psi\in\End(A)$. In the case $K=\Q$, an interesting type of
$\Q$-varieties are the \emph{building blocks} (namely, those whose
endomorphism algebra is a central division algebra over a totally
real number field $F$, with Schur index $t=1$ or $t=2$ and
$t[F:\Q]=\dim A$), since they are known to be the absolutely
simple  factors up to isogeny of the non-CM Abelian varieties of
$\GL_2$-type (see~\cite[Chapter 4]{Py}). After the validity of
Serre's conjecture~\cite[$3.2.4_?$]{Se} on representations of
$\Gal(\Qb/\Q)$, a theorem of Ribet (\cite[Theorem 4.4]{Ri})
implies that they are indeed the non-CM absolutely simple factors
up to isogeny of the modular Jacobians $J_1(N)$.

Elkies  proved in~\cite{El} that  non-CM Abelian $K$-varieties of
dimension one (also called elliptic $K$-curves) are parametrized
up to isogeny by the non-cusp  $K$-rational points without CM of
the curves $X_0(N)/W(N)$ for square-free positive integers $N$,
where $W(N)$ is the group of Atkin-Lehner involutions of $X_0(N)$.
In this note we adapt Elkies's original argument to prove an
analogous result for Abelian $K$-surfaces whose endomorphism
algebra is an indefinite rational quaternion algebra; in this
case, they are parametrized by the $K$-rational points of the
quotient of a Shimura curve by its group of Atkin-Lehner
involutions.

\section{Shimura curves, Atkin-Lehner involutions and isogenies}
\label{sec:Shimura varieties and Atkin-Lehner involutions} The aim
of this section is to recall the basic definitions and results
concerning  Shimura curves and  their interpretation as moduli
spaces of certain Abelian surfaces, and also to study the isogeny
class of such Abelian surfaces in this context. The presentation
of the background material is based mostly on the first chapter
of~\cite{BD}.

\subsection{ QM-Abelian surfaces and Shimura curves}

Let $B$ be an indefinite quaternion algebra over $\Q$ of
discriminant $D$, and let $\cO$ be an Eichler order in $B$ of
level $N$, which we suppose square-free and prime to $D$ (we allow
also the case $N=1$, in which $\cO$ is actually a maximal order).
An \emph{Abelian surface with QM by $\cO$} is a pair $(A,\imath)$,
where $A/\C$ is an Abelian surface and $\imath$ is an embedding
$\cO\hookrightarrow \End(A)$,  satisfying  that $H_1(A,\Z)\simeq
\cO$ as left $\cO$-modules (here the structure of $\cO$-module in
$H_1(A,\Z)$ is given by $\imath$). If the order $\cO$ is clear by
the context, we will call them just \emph{QM-Abelian surfaces}.

We will denote by $X(D,N)$ the Shimura curve defined over $\Q$
associated with the moduli problem of classifying isomorphism
classes of Abelian surfaces $(A,\imath)$ with QM by $\cO$. We
remark  that $X(D,N)$ can also be defined  as the moduli space for
quadruples $(A,\imath,P,Q_N)$, where $(A,\imath)$ is an Abelian
surface with QM by a maximal order $\cO_0$, $P$ is a principal
polarization on $A$ satisfying certain compatibility conditions
with $\imath$, and $Q_N$ is a level $N$-structure (i.e. a subgroup
of $A$ isomorphic to $\Z/N\Z\times \Z/N\Z$ and cyclic as
$\cO_0$-module). Actually, both moduli problems coincide. Indeed,
a result of Milne asserts that in this case there exists a unique
compatible principal polarization, so we can remove it from the
moduli problem. Moreover, considering the level $N$-structure is
equivalent to considering embeddings of the Eichler order
$\cO\hookrightarrow \End(A)$, and in this way we obtain the moduli
interpretation for $X(D,N)$ we started with, which is the one we
will use. The reader can consult~\cite[\S 0.3.2]{Cl} for more
details on the several moduli interpretations for $X(D,N)$.

Let $\hat \cO=\cO\otimes \hat\Z$, $\hat B=B\otimes \hat \Z$ and
$\P=\C \smallsetminus \R$. By fixing an isomorphism $B\otimes
\R\simeq \M_2(\R)$, we can make $B^\times$ act on $\P$ by
fractional linear transformations, and there is a canonical
identification $ X(D,N)(\C)\simeq \P/\cO^\times$. It associates to
$\tau\in \P$ the pair $(\C^2/\Lambda_\tau,\imath_\tau)$, where
$\Lambda_\tau$ is the lattice $ \cO\binom{1}{\tau}$ and
$\imath_\tau$ is the natural inclusion given by the action of
$\cO$ on $\C^2$. Starting from an Abelian surface $(A,\imath)$
with QM by $\cO$, we can write $A\simeq \C^2/\Lambda$ for some
lattice $\Lambda$. Since $\Lambda\simeq \cO$, and after scaling
the lattice if necessary, we have that $\Lambda\simeq
\cO\binom{1}{\tau}$ for some $\tau\in \P$, and this gives the
reverse map.

Since $B$ is indefinite we have that $\# (\hat\cO^\times\setminus
\hat B^\times/ B^\times)=1$, and this implies that
$$\P/\cO^\times\simeq (\hat\cO^\times \setminus \hat B^\times
\times \P)/B^\times\simeq (\hat\cO^\times \setminus \hat
B^\times/\Q^\times \times \P)/B^\times.$$ Note that the double
coset $\hat\cO^\times\backslash \hat B^\times/ \Q^\times$
represents the set of left fractional ideals in $B$, modulo the
relation given by the multiplication of the ideals by rational
numbers. Each fractional ideal $J$ can be multiplied by a rational
number to obtain an integral ideal $I$ which is contained in no
proper ideal of the form $k\cO$, with $k\in\Z$. In this way, we
can also identify $\hat\cO^\times\backslash \hat B^\times/
\Q^\times$ with the set of left integral ideals that are not
contained in any proper ideal of the form $k\cO$ with $k\in\Z$.
Therefore, any point in $X(D,N)(\C)$ can be represented by a pair
of the form $(I,\tau)$, where $I$ is a left ideal of $\cO$ and
$\tau$ belongs to $\C\smallsetminus\R$; it is easy to see that the
QM-Abelian surface corresponding to this point in the moduli
interpretation is
$\left(\C^2/I\binom{1}{\tau},\imath_\tau\right)$, were
$\imath_\tau$ is the natural inclusion given by the action of
$\cO$ on $\C^2$ (note that this gives a well defined action on
$\C^2/I\binom{1}{\tau}$ because $I$ is a left ideal).

Finally, since the class number of $\Q$ is $1$ we have that
$\Q^\times \hat\Z^\times=\hat\Q^\times$, and therefore we  also
have the identification $$X(D,N)\simeq (\hat\cO^\times \setminus
\hat B^\times/\hat\Q^\times \times \P)/B^\times.$$

\subsection{Trees and Atkin-Lehner involutions}
We have a decomposition into local double cosets
$$\hat\cO^\times\backslash \hat B^\times/\hat
\Q^\times\simeq{\prod_\ell}'\cO^\times_\ell\backslash
B_\ell^\times/\Q_\ell^\times,$$ where $\cO_\ell=\cO\otimes
\Z_\ell$, $B_\ell=B\otimes \Z_\ell$ and $\prod_\ell'$ denotes the
restricted product over all  primes. When $\ell\nmid DN$ then
$\cBl\simeq\PGL_2(\Z_\ell)\backslash \PGL_2(\Q_\ell)$, which is
identified with the set of  vertices of the Bruhat-Tits tree of
$\PGL_2(\Q_\ell)$, a regular tree of degree $\ell+1$. If $\ell\mid
N$ then, since we are assuming $N$ to be square-free, we have
$\cBl\simeq\Gamma_0(\ell)\backslash \PGL_2(\Q_\ell)$, which is
identified with the set of oriented edges of the Bruhat-Tits tree
of $\PGL_2(\Q_\ell)$. If $\ell\mid D$ then $\cBl$ has only two
elements, and we identify them with an oriented edge (that is,
each element corresponds to one orientation of the edge). Hence,
for $\ell\mid ND$ there is a natural involution on $\cBl$; namely,
the one that reverses  the orientation of the edges. This
involution extends to an \emph{Atkin-Lehner involution} $W_\ell$
on $X(D,N)$, and we denote by $W(D,N)=\langle W_\ell\, \colon \,
\ell\mid ND\rangle$. As usual, if $n\mid ND$ then $W_{n}$ stands
for the composition of all the $W_\ell$ with $\ell\mid n$.

A  maximal order $\cO_0$ such that $\cO\subseteq \cO_0$ gives rise
to a natural morphism $\phi: X(D,N)\ra X(D,1)$, which at the level
of complex points is the natural map $(\hat\cO^\times\backslash
\hat B^\times/\hat \Q^\times\times \P)/ B^\times\ra
(\hat\cO_0^\times\backslash \hat B^\times/\hat \Q^\times\times
\P)/ B^\times.$ We can use $\phi$ to define another morphism
$\varphi:X(D,N)\ra X(D,1)\times X(D,1)$ by means of
$\varphi(P)=(\phi(P),\phi(W_N(P)))$. If $(b_\ell)_\ell$ belongs to
$\prod_\ell'\cO^\times_\ell\backslash
B_\ell^\times/\Q_\ell^\times$ and $\tau$ is a non-real complex
number, then $((b_\ell)_\ell,\tau)$ represents a point in
$X(D,N)$. It is sent by $\phi$ to the point represented by
$((b_\ell')_\ell,\tau)$, where $b_\ell'=b_\ell$ for $\ell\nmid N$,
and $b_\ell'$ is the origin of the edge $b_\ell$ for $\ell\mid N$.
 This  interpretation of $\phi$ makes it clear the fact that
 $\varphi$ is injective.

\subsection{The isogeny class of a QM-Abelian surface}

An \emph{isogeny} between two QM-Abelian surfaces $(A,\imath)$ and
$(A',\imath')$ is an isogeny $\mu:A\ra A'$ that respects the
action of $\cO$, i.e. such that
$\imath'(\psi)\circ\mu=\mu\circ\imath(\psi)$ for all $\psi$ in
$\cO$. We will denote by $[(A,\imath)]$, or for ease of notation
just by $[A,\imath]$ in some cases, the isogeny class of
$(A,\imath)$; that is, the set of all QM-Abelian surfaces
isogenous to $(A,\imath)$ up to isomorphism. In the following
lemma we characterize $[A,\imath]$ as a subset of $X(D,N)(\C)$:\\

\begin{lemma}
Let $I\subseteq \cO$ be a left ideal, $\tau$ a non-real complex
number and let  $(A,\imath)$ be the point in $X(D,N)(\C)$
represented by $(I,\tau)$. Then $
[A,\imath]=(\hat\cO^\times\backslash\hat B^\times/\Q^\times\times
\{ \tau \})/B^\times\subseteq X(D,N)(\C)$. Moreover, if
$(A,\imath)$ does not have CM then we can identify $[A,\imath]$
with $\hat\cO^\times\backslash\hat B^\times/\Q^\times\times \{
\tau \}$.
\end{lemma}
\begin{proof}
First of all, we claim that there is a one-to-one correspondence
between the isogenies $(A',\imath')\ra (A,\imath)$ of degree $n^2$
and the left ideals of $\cO$ of norm $n$. Indeed, if we write
$A\simeq\C^2/H_1(A,\Z)$ and $A'\simeq\C^2/H_1(A',\Z)$, giving an
isogeny $A'\ra A$ is equivalent to giving an inclusion
$H_1(A',\Z)\subseteq H_1(A,\Z)\simeq \cO$, and the condition on
the isogeny to be compatible with the action of $\cO$ translates
into the condition on $H_1(A',\Z)$ to be a left ideal of $\cO$. In
addition, if the degree of the isogeny is $n^2$, then
$\#\cO/H_1(A',\Z)=n^2$, and therefore the norm of the ideal
$H_1(A',\Z)$ is equal to $n$. This proves the claim. Now, we
observe that ideals of the form $k\cO$ for some $k\in\Z$ give rise
to isogenies $(A',\imath')\ra (A,\imath)$ with $(A',\imath')\simeq
(A,\imath)$, because they correspond to the isogenies
`multiplication by $k$' in $(A,\imath)$. Since
$\hat\cO^\times\backslash \hat B^\times/\Q^\times$ is the set of
all left ideals not contained in any proper ideal of the form
$k\cO$ with $k\in \Z$, we see that the orbit
$(\hat\cO^\times\backslash \hat B^\times/\Q^\times\times \{ \tau
\})/B^\times\subseteq X(D,N)(\C)$ contains a representative for
each $(A',\imath')$ isogenous to $(A,\imath)$. If $(A,\imath)$
does not have CM this orbit can be identified with
$\hat\cO^\times\backslash\hat B^\times/\Q^\times\times \{ \tau
\}$, because then any element  $b\in B^\times$  such that
$b\tau=\tau$ necessarily belongs to $ \Q^\times$.
\end{proof}
\begin{corollary}\label{cor}
If $(A,\imath)$ does not have CM, then $\phi([A,\imath])\subseteq
[\phi(A,\imath)]$ and $W_n([A,\imath])=[A,\imath]$ for all $n\mid
ND$.
\end{corollary}
\begin{proof}
By  the lemma and by the description of $\phi$ at the level of
complex points we have that
$$\phi([A,\imath])=\phi(\hat\cO^\times\backslash \hat B^\times/
\Q^\times\times \{ \tau \})\subseteq \hat\cO_0^\times\backslash
\hat B^\times/ \Q^\times\times \{ \tau \}=[\phi(A,\imath)].$$ By
the definition of $W_n$ we see that
$$W_n([A,\imath])=W_n(\hat\cO^\times\backslash \hat B^\times/
\Q^\times\times \{ \tau \})=\hat\cO^\times\backslash \hat
B^\times/ \Q^\times\times \{ \tau \}=[A,\imath],$$ and this gives
the second statement.
\end{proof}

\section{$K$-surfaces with QM}
\label{Q-surfaces with QM}  Let $\cO$ be an Eichler order of
square-free level $N$ in an indefinite rational quaternion algebra
of discriminant $D$, and let $\cO_0$ be a maximal order  with
$\cO\subseteq \cO_0$. Let $P$ be a non-CM $K$-rational point in
$X(D,N)/W(D,N)$. A preimage $Q\in X(D,N)$ of $P$ under the
quotient map corresponds to an Abelian surface $(A,\imath)/\Kb$
with QM by $\cO$, and $\phi(Q)$ corresponds to an Abelian surface
$(A_0,\imath_0)/\Kb$ with QM by $\cO_0$. For each
$\sigma\in\Gal(\Kb/K)$ there exists an integer $n\mid ND$ such
that $\acc\sigma Q=W_n(Q)$, and therefore
$\acc\sigma\phi(Q)=\phi(\acc\sigma Q)=\phi(W_n(Q))$. But
$\acc\sigma \phi(Q)$ corresponds to $(\acc\sigma
A_0,\acc\sigma\imath_0)$, and by corollary~\ref{cor} we have that
$\phi(W_n(Q)) $ belongs to $[A_0,\imath_0]$. This means that
$(A_0,\imath_0)$ and $(\acc\sigma A_0,\acc\sigma\imath_0)$ are
isogenous for all $\sigma\in\Gal(\Kb/K)$; we say that
$(A_0,\imath_0)$ is an \emph{Abelian $K$-surface with QM by
$\cO_0$}. We also have the following converse to this
construction:\\

\begin{theorem}
Let $(A_0,\imath_0)$ be a non-CM  Abelian $K$-surface with QM by
the maximal order $\cO_0$. Then there exists a square-free $N$,
depending only on the isogeny class of $(A_0,\imath_0)$, such that
$(A_0,\imath_0)$ is isogenous to the QM-Abelian surface obtained
by the above procedure applied to some $K$-rational point in
$X(D,N)/W(D,N)$. Moreover,  if a $K$-rational point of
$X(D,N')/W(D,N')$ parametrizes an Abelian $K$-surface with QM
isogenous to $(A_0,\imath_0)$, then $N\mid N'$.
\end{theorem}
\begin{proof}
The pair $(A_0,\imath_0)$ gives a point in $X(D,1)$, that we can
represent by $((b_\ell)_\ell,\tau)$ for some
$(b_\ell)_\ell\in\prod_\ell' \cO_{0\ell}^\times\backslash
B_\ell^\times/\Q_\ell^\times$ and some complex number $\tau$.
Recall that $[A_0,\imath_0]=\prod_\ell'
\cO_{0\ell}^\times\backslash B_\ell^\times/\Q_\ell^\times\times \{
\tau \}$, and that for $\ell\nmid D$ we identify $
\cO_{0\ell}^\times\backslash B_\ell^\times/\Q_\ell^\times$ with an
homogenous tree of degree $\ell+1$: its vertexes are the pairs
$(A'_0,\imath'_0)$ isogenous to $(A_0,\imath_0)$ with an isogeny
of degree a power of $\ell$, and two vertexes are connected if
there exists an isogeny of degree $\ell^2$ between them. Also for
$\ell\mid D$ we identify $\cO_{0\ell}^\times\setminus
B_\ell^\times/\Q_\ell^\times$ with an oriented edge. Denote by
$\pi_\ell$ the projection $[A_0,\imath_0]\ra
\cO_{0\ell}^\times\backslash B_\ell^\times/\Q_\ell^\times$, and by
$\langle A_0,\imath_0\rangle$ the finite  set of Abelian surfaces
up to isomorphism with QM by $\cO_0$ that are
$\Gal(\Kb/K)$-conjugated to $(A_0,\imath_0)$. Note that $\langle
A_0,\imath_0\rangle\subseteq [A_0,\imath_0]$, because
$(A_0,\imath_0)$ is an Abelian $K$-surface with QM. 

We consider
the action of $\Gal(\Kb/K)$ on $\pi_\ell \langle
A_0,\imath_0\rangle$ defined by
$\acc\sigma(\pi_\ell(B,\jmath))=\pi_\ell(\acc\sigma
B,\acc\sigma\jmath)$. Note that $\pi_\ell\langle
A_0,\imath_0\rangle$ will contain a single vertex for all but
finitely many primes $\ell$. Following Elkies, for each $\ell$ we
construct an edge or a vertex of $\pi_\ell \langle
A_0,\imath_0\rangle$ fixed by $\Gal(\Kb/K)$: it is the central
vertex or edge of any path of maximum length  joining two vertexes
in $\pi_\ell \langle A_0,\imath_0\rangle$, and we will call it the
\emph{center} of $\pi_\ell\langle A_0,\imath_0\rangle$. It is a
well known property of trees that this vertex or edge does not
depend on the  path chosen, and since  $\Gal(\Kb/K)$ takes one
path of maximum length to another, it is clear that the center is
fixed by $\Gal(\Kb/K)$. Define $N$ to be the product of all the
primes $\ell\nmid D$ such that the center of $\pi_\ell\langle
A_0,\imath_0\rangle$ is an edge, and let $\cO$ be an Eichler order
of level $N$. The fact that $\pi_\ell[A_0,\imath_0]$ is a  tree
implies that if the center of $\pi_\ell\langle
A_0,\imath_0\rangle$ is an edge, then it is necessarily the only
edge or vertex in $\pi_\ell[A_0,\imath_0]$ fixed by $\Gal(\Kb/K)$
(otherwise there would exist a cycle; this can be  seen by
considering the action on the tree of a $\sigma\in \Gal(\Kb/K)$
that swaps the vertices of the central edge). Thus any pair in
$[A_0,\imath_0]$ produces the same $N$. 

For each $\ell\mid N$,
choose an orientation of the center and call $b_\ell'$ this
oriented edge in the graph $\PGL_2(\Z_\ell)\backslash
\PGL_2(\Q_\ell)$; recall that we can identify $b_\ell'$ with an
element of $ \cO_{\ell}^\times\backslash
B_\ell^\times/\Q_\ell^\times$. For each $\ell\nmid N$ let
$b_\ell'=b_\ell$, but viewed as an element in $
\cO_{\ell}^\times\backslash B_\ell^\times/\Q_\ell^\times$. Now the
pair $\left((b_\ell')_\ell,\tau\right)$ defines  a point
$Q=(A,\imath)\in X(D,N)(\Kb)$, with the property that $\phi(Q)\in
[A_0,\imath_0]$.  If we represent $\phi(Q)$ by
$((c_\ell)_\ell,\tau)$ and $\acc\sigma\phi(Q)$ by
$((c_\ell')_\ell,\tau)$, then $c_\ell=c_\ell'$ for all $\ell\nmid
ND$. But for some $\ell\mid N$,  $c_\ell'$ can be the vertex of
the center of $\pi_\ell\langle A_0,\imath_0\rangle$ which is
different from $c_\ell$, and for some $\ell\mid D$, $c_\ell'$ and
$c_\ell$ can have opposite orientation. If $n$ is the product of
the primes $\ell$ where $c_\ell$ and $c_\ell'$ differ, we have
that $\acc\sigma\phi(Q)=\phi(W_n(Q))$, which implies that
$\phi(\acc\sigma Q)=\phi(W_n(Q))$. A similar argument shows that
$\phi(W_N(\acc\sigma Q))=\phi(W_NW_n(Q))$. Therefore, by the
injectivity of the map $\varphi$ defined in the
Section~\ref{sec:Shimura varieties and Atkin-Lehner involutions}
we have that $\acc\sigma Q=W_n Q$, and the image of $Q$ by the
quotient map $X(D,N)\ra X(D,N)/W(D,N)$ is a $K$-rational point
$P$. Since $\phi(Q)\in [A_0,\imath_0]$, it is clear that applying
to $P$ the process for obtaining an Abelian $K$-surface with QM
described at the beginning of the section, we obtain a pair
isogenous to $(A_0,\imath_0)$. 

Finally, to see the last statement
in the theorem, note that if $\ell\mid N$ and $(A_0',\imath_0')$
comes from a $K$-rational point in $X(D,N')/W(D,N')$ for some $N'$
not divisible by $\ell$, then $\pi_\ell\langle A_0',\imath_0'
\rangle$ would be a  vertex fixed by $\Gal(\Kb/K)$, which is not
possible because $\pi_\ell [A_0,\imath_0]$ contains an edge fixed
by $\Gal(\Kb/K)$.
\end{proof}

\begin{remark} For all but finitely many values of $D$ and $N$ the curve
$X(D,N)/W(D,N)$ has genus at least 2 (see~\cite[Corollary 50
]{Cl}), and therefore it  has  a finite number of $K$-rational
points.
\end{remark}
\begin{remark}
 So far in this section we have seen that a $K$-rational point in
 the quotient
$X(D,N)/W(D,N)$ produces an Abelian $K$-surface $(A_0,\imath_0)$
with QM by a maximal order $\cO_0$, and that any such pair arises
from a $K$-rational point in $X(D,N)/W(D,N)$ for some square-free
$N$. This result  in fact gives slightly more information than the
strictly needed in the setting we described in the introduction;
hence, if we are interested only in Abelian $K$-surfaces with QM
up to isogeny, and we do not care about the precise embedding
$\imath$, we  can just forget  this information. Indeed, a
$K$-rational point in $X(D,N)/W(D,N)$ gives rise to an Abelian
$K$-surface $A$ whose endomorphism algebra is isomorphic to the
quaternion algebra over $\Q$ of discriminant $D$. Conversely,
given an Abelian $K$-surface with endomorphism algebra isomorphic
to the quaternion algebra over $\Q$ of discriminant $D$, we can
find in its isogeny class a variety $A_0$ such that there exists
an embedding $\imath_0:\cO_0\hookrightarrow \End(A_0)$ for some
maximal order $\cO_0$. Then this variety is in turn isogenous to
one arising from a $K$-rational point in $X(D,N)/W(D,N)$ for some
$N$.
\end{remark}

\section*{Acknowledgements}
We would like to thank Josep Gonz\'alez, Jordi Quer and V\'ictor
Rotger for their helpful comments on an earlier version of the
manuscript. We are also thankful to the referee, whose suggestions
helped to improve the presentation of the article.

\end{document}